\documentclass[12pt,oneside,english]{amsart}
\usepackage[T1]{fontenc}
\usepackage[latin9]{inputenc}
\usepackage{geometry}
\usepackage{babel}
\usepackage{amsthm}
\usepackage{amssymb}
\usepackage{esint}
\usepackage[unicode=true,pdfusetitle,
 bookmarks=true,bookmarksnumbered=false,bookmarksopen=false,
 breaklinks=false,pdfborder={0 0 1},backref=false,colorlinks=false]
 {hyperref}

\makeatletter
\numberwithin{equation}{section}
\numberwithin{figure}{section}
\theoremstyle{plain}
\newtheorem{thm}{\protect\theoremname}
  \theoremstyle{definition}
  \newtheorem{defn}[thm]{\protect\definitionname}
  \theoremstyle{remark}
  \newtheorem{rem}[thm]{\protect\remarkname}
  \theoremstyle{plain}
  \newtheorem{lem}[thm]{\protect\lemmaname}
  \theoremstyle{remark}
  \newtheorem*{acknowledgement*}{\protect\acknowledgementname}

\makeatother

  \providecommand{\acknowledgementname}{Acknowledgement}
  \providecommand{\definitionname}{Definition}
  \providecommand{\lemmaname}{Lemma}
  \providecommand{\remarkname}{Remark}
\providecommand{\theoremname}{Theorem}

\begin{document}

\title{Stability and error estimates of a general modified quasi-boundary
value method via a semi-linear backward parabolic equation}

\author{Vo Anh Khoa}

\address{Mathematics and Computer Science Division, Gran Sasso Science Institute,
Viale Francesco Crispi 7, 67100 L'Aquila, Italy}

\email{khoa.vo@gssi.infn.it, vakhoa.hcmus@gmail.com}


\begin{abstract}
Regularization methods have been recently developed to construct stable
approximate solutions to classical partial differential equations
considered as final value problems. In this paper, we investigate
the backward parabolic problem with locally Lipschitz source: $\partial_{t}u+\mu\left(t\right)\mathcal{A}u\left(t\right)=f\left(t,u\right)$
where $\mathcal{A}:\mathcal{D}\left(\mathcal{A}\right)\subset\mathcal{H}\to\mathcal{H}$
is a positive, self-adjoint and unbounded linear operator on the Hilbert
space $\mathcal{H}$. The problem arises in many applications, but
it is in general ill-posed. The ill-posedness is caused by catastrophic
growth in the representation of solution, then independence of solution
on data makes computational procedures impossible. Therefore, we contribute
to this interesting field the study of the stable approximation solution
via the modified quasi-boundary value method based on nonlinear spectral
theory, which slightly extends the results in many works. Our main
focus is to qualitatively prove the stability and convergence rate
where semi-group technicalities are highly useful.
\end{abstract}

\maketitle

\section{Introduction}

Let $\mathcal{H}$ be an abstract Hilbert space endowed with the inner
product $\left\langle \cdot,\cdot\right\rangle $ and the norm $\left\Vert \cdot\right\Vert $.
Let $\mathcal{A}:\mathcal{D}\left(\mathcal{A}\right)\subset\mathcal{H}\to\mathcal{H}$
be a positive-definite self-adjoint unbounded linear operator such
that $-\mathcal{A}$ generates a compact contraction semi-group on
$\mathcal{H}$ and let $f:\left[0,T\right]\times\mathcal{H}\to\mathcal{H}$
for $T>0$. We consider the problem of finding a function $u:\left[0,T\right]\to\mathcal{H}$
such that
\begin{equation}
\begin{cases}
\partial_{t}u+\mu\left(t\right)\mathcal{A}u\left(t\right)=f\left(t,u\right), & t\in\left(0,T\right),\\
u\left(T\right)=u_{T}.
\end{cases}\label{eq:1}
\end{equation}

Our motivation is involved in many models where $\mathcal{A}=-\Delta$
is considered. For instance, taking $R\left(u\right)=u-u^{2}$ gives
us the Fisher equation \cite{Fis37} that was originally used to
describe the spreading of biological populations of diploid individuals
living in a one-dimensional habitat while interpreting the reaction
term as due to true births and deaths. One may also deduce another
equations: the Newell-Whitehead-Segel equation \cite{NW69} with $f\left(u\right)=u-u^{3}$
describes Rayleigh-Benard convection, and the Zeldovich equation \cite{Zel48}
$f\left(u\right)=u^{2}-u^{3}$ performing the temperature in combustion
theory. Additionally, it can be generalized into the Nagumo equation
\cite{NYA65} with $f\left(u\right)=u\left(1-u\right)\left(u-C\right)$
for $C\in\left(0,1\right)$ modeling the transmission of electrical
pulses in a nerve axon.

In principle, the problem \eqref{eq:1} is well known to be severely
ill-posed in the sense of Hadamard \cite{Ha23}, which could be said
that the solution (if it exists) does not depend continuously on the
data in any reasonable topology. Therefore, we would like to contribute
a stable approximation approach to deal with that obstacle.

In the past, many approaches have been studied for solving various
kinds of the PDE in \eqref{eq:1} and related problems, see \cite{BBC85,BR07,HDS08,HL14,HL15,LL67,Sei96,TQTT13-1,TQTT13,Ya99}
and the references therein. We notably mention that over the last
20 years or so, the modified quasi-boundary method has developed step
by step to maturity, allowing effective treatment of a broad range
of linear problems backward in time. In fact, in the literature of
such a method, it should be starting from the original establishment
by Showalter \emph{et al.} \cite{Show83}. In that work, he introduced
a way named quasi-boundary-value method to approximate the problem
\[
\begin{cases}
\partial_{t}u+\mathcal{A}u\left(t\right)-\mathcal{B}u\left(t\right)=0, & t\in\left(0,T\right),\\
u\left(0\right)=u_{0},
\end{cases}
\]

by adding the final condition coupled with a factoring parameter $\alpha>0$
small enough (i.e. $u\left(0\right)+\alpha u\left(T\right)=u_{0}$).

Inheriting from that idea, Clark and Oppenheimer \cite{CO94} in 1994
perturbed the final condition $\alpha u\left(0\right)+u\left(T\right)=u_{T}$
to approximate the backward problem
\[
\begin{cases}
\partial_{t}u+\mathcal{A}u\left(t\right)=0, & t\in\left(0,T\right),\\
u\left(T\right)=u_{T}.
\end{cases}
\]

As a result, they obtained the well-posedness of the approximate nonlocal
problem. Particularly, it can be proved that the stable approximate
solution converges uniformly in time to the classical solution. There
is also a well-known approach of perturbing the final observation
in such a way that $u\left(T\right)-\alpha u'\left(0\right)=u_{T}$,
proposed in \cite{DB05} by Denche and Bessila within the framework
of using a variant of that method to give an error estimate of logarithmic
type at $t=0$. In general, that result is better than the quasi-reversibility
method.

In the literature on the non-homogeneous case, Denche \emph{et al.}
\cite{DA12} gave extensions of the quasi-boundary methods while showing
several sharp error estimates including explicit convergence rates.
In particular, the authors constructed the stable approximate problem
that generally perturbs both the source term $f$ and the final value
$u_{T}$ by means of
\[
f_{\alpha}=\sum_{p=1}^{\infty}\frac{e^{-\lambda_{p}T}}{\alpha\lambda_{p}^{k}+e^{-\lambda_{p}T}}\left\langle f,\varphi_{p}\right\rangle \varphi_{p},\quad u_{T}^{\alpha}=\sum_{p=1}^{\infty}\frac{e^{-\lambda_{p}T}}{\alpha\lambda_{p}^{k}+e^{-\lambda_{p}T}}\left\langle u_{T},\varphi_{p}\right\rangle \varphi_{p},
\]

where $\left\{ \varphi_{p}\right\} _{p\in\mathbb{N}}$ is the set
of orthonormal basis generated by the operator $\mathcal{A}$. Then
each assumption on the data leads to a specific error estimate. For
example, if $\left\Vert \mathcal{A}^{s}u_{T}\right\Vert ^{2}<\infty$
holds, the convergence rate is of order $\mathcal{O}\left(\ln^{s}\left(T/s\alpha\right)\right)$.

The consideration of the specific problems in the non-homogeneous
case is also in conjunction with the backward heat problem. Since
2008, the authors Trong and Tuan in \cite{TT08} applied the method
(generalized by Denche \emph{et al.} \cite{DA12} in 2012) to solve
the problem,
\[
\begin{cases}
u_{t}-u_{xx}=f\left(x,t\right), & \left(x,t\right)\in\left(0,\pi\right)\times\left(0,T\right),\\
u\left(x,T\right)=g\left(x\right), & x\in\left(0,\pi\right),\\
u\left(0,t\right)=u\left(\pi,t\right)=0, & t\in\left(0,T\right).
\end{cases}
\]

It thus comes out the approximate problem
\[
\begin{cases}
u_{t}-u_{xx}=\sum_{p=1}^{\infty}\frac{e^{-Tp^{2}}}{\varepsilon p^{2}+e^{-Tp^{2}}}f_{p}\left(t\right)\sin\left(px\right), & \left(x,t\right)\in\left(0,\pi\right)\times\left(0,T\right),\\
u\left(x,T\right)=\sum_{p=1}^{\infty}\frac{e^{-Tp^{2}}}{\varepsilon p^{2}+e^{-Tp^{2}}}g_{p}\sin\left(px\right), & x\in\left(0,\pi\right),\\
u\left(0,t\right)=u\left(\pi,t\right)=0, & t\in\left(0,T\right),
\end{cases}
\]

where $\varepsilon\in\left(0,1\right)$ represents the error induced
by measurement, and $f_{p},g_{p}$ are, respectively, defined by the
inner product between $f,g$ and $\sin\left(px\right)$ over $L^{2}\left(0,\pi\right)$.

The development of regularization methods has also set the stage for
addressing more difficult problems with more challenging features.
As far as we know, the regularization approaches concerning problems
which contain nonlinear source term are, up to now, still limited
though that problem has a great position in many types of application
(\cite{LMKS93,ZJS11}). In recent works, this theme has been tackled
with source terms where they are globally Lipschitzian (\emph{e.g.}
\cite{TQ11}) or locally Lipschitzian with a special form (Tuan \emph{et
al.} \cite{TT14}). Moreover, we observe that most of the previous
works in this field did not take care of the time-dependent diffusion
much, especially our problem \eqref{eq:1}. To our knowledge, the
literature on such a problem is eventually narrow, and as might be
expected, only two papers currently fulfilled related problems (see,
\cite{TQTT13-1,TQTT13}). Motivated by those things, it is thus natural
to raise the following question. Whether the modified quasi-boundary
method based on \cite{DA12} to \eqref{eq:1} still works? As a result,
this paper marks an extensive foray towards the development of numerical
methods based on nonlinear spectral theory. Hence, the purpose of
this generalization here is to attain general degrees of convergence
while claiming that the method is successfully applicable to our problem.

The rest of the paper is organized as follows. At first, for ease
of the reading, we summarize in Section 2 some rudiments in semi-group
of operators and elementary results. In addition, we present the formal
solution of \eqref{eq:1} leading to the state of ill-posedness in
accordance with the construction of the regularized solution. We then
obtain in Section 3 well-posedness and convergence result. In the
last section, some concluding remarks are mentioned.

\section{Preliminaries}

We below present the functional settings together with notations,
which will be used in this paper. Afterwards, we derive the representation
of exact solution, and construct the regularized solution from those
materials. Such elementary materials and further details can be found
in \cite{BR07,Pazy83}. Therefore, we only list them and the results
are presented while skipping the proofs.

We denote by $\left\{ E_{\lambda},\lambda>0\right\} $ the spectral
resolution of the identify associated to $\mathcal{A}$. Let us denote
\[
\mathcal{S}\left(t\right)=e^{-t\mathcal{A}}=\int_{0}^{\infty}e^{-t\lambda}dE_{\lambda}\in\mathcal{L}\left(\mathcal{H}\right),\quad t\ge0,
\]

the $C_{0}$ semi-group of contractions generated by $-\mathcal{A}$
where $\mathcal{L}\left(\mathcal{H}\right)$ stands for the Banach
algebra of bounded linear operators on $\mathcal{H}$. Then, we define
\begin{equation}
\mathcal{A}u:=\int_{0}^{\infty}\lambda dE_{\lambda}u,\label{eq:2}
\end{equation}

for all $u\in\mathcal{D}\left(\mathcal{A}\right)$. In this connection,
$u\in\mathcal{D}\left(\mathcal{A}\right)$ if and only if the integral
\eqref{eq:2} exists, i.e.,
\[
\int_{0}^{\infty}\lambda^{2}d\left\Vert E_{\lambda}u\right\Vert ^{2}<\infty.
\]

For this family of operators $\left\{ \mathcal{S}\left(t\right)\right\} _{t\ge0}$,
we have:
\begin{enumerate}
\item $\left\Vert \mathcal{S}\left(t\right)\right\Vert \le1$ for all $t\ge0$;
\item the function $t\mapsto\mathcal{S}\left(t\right),t\ge0$ is analytic;
\item for every real $r\ge0$ and $t>0$, the operator $\mathcal{S}\left(t\right)\in\mathcal{L}\left(\mathcal{H},\mathcal{D}\left(\mathcal{A}^{r}\right)\right)$; 
\item for every integer $k\ge0$ and $t>0$, $\left\Vert \mathcal{S}^{k}\left(t\right)\right\Vert =\left\Vert \mathcal{A}^{k}\mathcal{S}\left(t\right)\right\Vert \le c\left(k\right)t^{-k}$;
\item for every $x\in\mathcal{D}\left(\mathcal{A}^{r}\right),r\ge0$, we
have $\mathcal{S}\left(t\right)\mathcal{A}^{r}x=\mathcal{A}^{r}\mathcal{S}\left(t\right)x$.\end{enumerate}
\begin{defn}
\label{def:1}Let $f:\mathbb{R}\to\mathcal{H}$ be a piecewise continuous
function, we set
\[
\mathcal{D}\left(f\left(\mathcal{A}\right)\right):=\left\{ u\in\mathcal{H}:\int_{0}^{\infty}\left|f\left(\lambda\right)\right|^{2}d\left\Vert E_{\lambda}u\right\Vert ^{2}<\infty\right\} ,
\]

and define the operator $f\left(\mathcal{A}\right):\mathcal{D}\left(\mathcal{A}\right)\subset\mathcal{H}\to\mathcal{H}$
by the formula
\[
f\left(\mathcal{A}\right)u:=\int_{0}^{\infty}f\left(\lambda\right)dE_{\lambda}u,
\]

for all $u\in\mathcal{D}\left(f\left(\mathcal{A}\right)\right)$.
\end{defn}
In this work, we assume:

$\left(\mbox{A}_{1}\right)$ $u_{T}\in\mathcal{H}$ the given final
state of observation;

$\left(\mbox{A}_{2}\right)$ $\mu:\left[0,T\right]\to\mathbb{R}$
is differentiable, plays a role as time-dependent function such that
there exists $0<p\le q$, $\mu\left(t\right)\in\left[p,q\right]$
for all $t\in\left[0,T\right]$;

$\left(\mbox{A}_{3}\right)$ there exists a function, denoted by $u_{T}^{\delta}$,
depending on $\mathcal{H}$ satisfies $\left\Vert u_{T}^{\delta}-u_{T}\right\Vert \le\delta$
for $\delta\in\left(0,1\right)$ representing a bound on measurement
errors;

$\left(\mbox{A}_{4}\right)$ there is a positive constant $C_{R}<\infty$
defined by
\[
C_{R}:=\sup\left\{ \left|\frac{f\left(t,u\right)-f\left(t,v\right)}{u-v}\right|:\left|u\right|,\left|v\right|\le R,u\ne v,t\in\left[0,T\right]\right\} ,
\]

then $C_{R}$ is a non-decreasing function of $R$ and
\[
\left|f\left(t,u\right)-f\left(t,v\right)\right|\le C_{R}\left|u-v\right|,
\]

for every $R>0$, $\left|u\right|,\left|v\right|\le R$ and $t\in\left[0,T\right]$.

Now we are ready to state the representation of exact solution of
\eqref{eq:1}. Assume that the problem \eqref{eq:1} admits a unique
solution, then this solution can be represented by
\begin{equation}
u\left(t\right)=\mathcal{S}^{-1}\left(\bar{\mu}\left(t,T\right)\right)u_{T}-\int_{t}^{T}\mathcal{S}\left(\bar{\mu}\left(s,t\right)\right)f\left(s,u\right)ds,\label{eq:3}
\end{equation}

where the function $\bar{\mu}$ is defined by
\[
\bar{\mu}\left(a,b\right):=\int_{a}^{b}\mu\left(s\right)ds,\quad0\le a\le b\le T.
\]

\begin{rem}
The inverse operator $\mathcal{S}^{-1}\left(t\right)$ of $\mathcal{S}\left(t\right)$
can be defined for every $t>0$ since $\mathcal{S}\left(t\right)$
is self-adjoint and one-to-one with range dense, i.e. $\mathcal{S}\left(t\right)=\mathcal{S}^{*}\left(t\right),\ker\left(\mathcal{S}\left(t\right)\right)=\left\{ 0\right\} $
and $\overline{\mathcal{R}\left(\mathcal{S}\left(t\right)\right)}=\mathcal{H}$.
Thus, the inverse operator is $\mathcal{S}^{-1}\left(t\right)=e^{t\mathcal{A}}$.
Furthermore, since the semi-group $\mathcal{S}\left(t\right)$ is
analytic with the following property
\[
\mathcal{R}\left(\mathcal{S}\left(t\right)\right)\subset\bigcap_{n=1}^{\infty}\mathcal{D}\left(\mathcal{A}^{n}\right)\subset\mathcal{H},\quad t\in\left(0,T\right],
\]

and coupling with $\mathcal{R}\left(\mathcal{S}\left(t\right)\right)\ne\overline{\mathcal{R}\left(\mathcal{S}\left(t\right)\right)}=\mathcal{H}$,
we state that $\mathcal{R}\left(\mathcal{S}\left(t\right)\right)$
is not closed. Therefore, the inverse operator is unbounded in $\mathcal{H}$
for $t\in\left(0,T\right]$. As a consequence, let us say that $\mathcal{S}^{-1}\left(\bar{\mu}\left(t,T\right)\right)=e^{\bar{\mu}\left(t,T\right)\mathcal{A}}$
and $\mathcal{S}\left(\bar{\mu}\left(s,t\right)\right)=e^{-\bar{\mu}\left(s,t\right)\mathcal{A}}$,
then the representation of solution \eqref{eq:3} becomes
\begin{eqnarray}
u\left(t\right) & = & e^{\bar{\mu}\left(t,T\right)\mathcal{A}}u_{T}-\int_{t}^{T}e^{-\bar{\mu}\left(s,t\right)\mathcal{A}}f\left(s\right)ds\nonumber \\
 & = & \int_{0}^{\infty}e^{\bar{\mu}\left(t,T\right)\lambda}dE_{\lambda}u_{T}-\int_{t}^{T}\int_{0}^{\infty}e^{\bar{\mu}\left(t,s\right)\lambda}dE_{\lambda}f\left(s,u\right)ds.\label{eq:4}
\end{eqnarray}

Now we can see that the rapid escalation of exponential functions,
$e^{\bar{\mu}\left(t,T\right)\lambda}$ and $e^{\bar{\mu}\left(t,s\right)\lambda}$
for $t\le s\le T$, causes the ill-posedness of the problem, then
performing classical calculation is definitely impossible. Therefore,
contributing a regularization method is necessary to make the numerical
computation possible.
\end{rem}
On the one hand, for $R>0$, we define the truncated function $f_{R}$
such that
\begin{equation}
f_{R}\left(t,u\right):=\begin{cases}
f\left(t,R\right), & u>R,\\
f\left(t,u\right), & \left|u\right|\le R,\\
f\left(t,-R\right), & u<-R.
\end{cases}\label{eq:5}
\end{equation}

Then $f_{R}\in L^{\infty}\left(\left[0,T\right]\times\mathbb{R}\right)$
and it satisfies
\begin{equation}
\left|f_{R}\left(t,u\right)-f_{R}\left(t,v\right)\right|\le C_{R}\left|u-v\right|,\label{eq:2.5}
\end{equation}

for all $u,v\in\mathbb{R}$ and $t\in\left[0,T\right]$.

On the other hand, we introduce a filtering function such that for
each $\delta>0$ and given $k\ge1$ then $\Phi_{\bar{\mu},\lambda}^{\delta}:\left[0,T\right]\times\left[0,T\right]\to\mathbb{R}$
satisfies
\begin{equation}
\Phi_{\bar{\mu},\lambda}^{\delta}\left(s,t\right):=\frac{e^{\left(\bar{\mu}\left(t,s\right)-\bar{\mu}\left(0,T\right)\right)\lambda}}{\delta\lambda^{k}+e^{-\bar{\mu}\left(0,T\right)\lambda}},\quad0\le t\le s\le T.\label{eq:6}
\end{equation}

Combining \eqref{eq:5} and \eqref{eq:6}, and observe \eqref{eq:4},
the approximate solution, denoted by $u^{\delta}$, corresponding
to noise data can be deduced as follows:
\begin{equation}
u^{\delta}\left(t\right):=\int_{0}^{\infty}\Phi_{\bar{\mu},\lambda}^{\delta}\left(T,t\right)dE_{\lambda}u_{T}^{\delta}-\int_{t}^{T}\int_{0}^{\infty}\Phi_{\bar{\mu},\lambda}^{\delta}\left(s,t\right)dE_{\lambda}f_{R^{\delta}}\left(s,u^{\delta}\right)ds,\label{eq:7}
\end{equation}

for $0\le t\le T$. Here $R^{\delta}:=R\left(\delta\right)>0$ tends
to infinity as $\delta\to0^{+}$.

Now we shall prove some elementary inequalities to support our proofs
in Section 3.
\begin{lem}
\label{lem:1}Given $k\ge1$, let $\delta$ and $M$ be positive constants,
we have the following inequality
\[
\frac{1}{\delta x^{k}+e^{-Mx}}\le\frac{1}{\delta}\left(\frac{kM}{\ln\left(\frac{M^{k}}{k\delta}\right)}\right)^{k},\quad x>0.
\]
\end{lem}
\begin{proof}
Let us define $g\left(x\right)=\dfrac{1}{\delta x^{k}+e^{-Mx}}$,
then taking its derivative implies
\begin{equation}
g'\left(x\right)=\frac{Me^{-Mx}-\delta kx^{k-1}}{\left(\delta x^{k}+e^{-Mx}\right)^{2}}.\label{eq:15}
\end{equation}

The equation ${\displaystyle g'\left(x\right)=0}$ gives a unique
solution $x_{0}$ such that the numerator of the right-hand side of
\eqref{eq:15} vanishes at $x_{0}$, i.e. $x_{0}^{k-1}e^{Mx_{0}}=\dfrac{M}{k\delta}$.
The function $g$ thus achieves its maximum at $x=x_{0}$. It has
also been shown that $e^{-Mx_{0}}=\dfrac{k\delta}{M}x_{0}^{k-1}$,
so the function $g$ can be bounded by the following
\begin{equation}
g\left(x\right)\le\frac{1}{\delta x_{0}^{k}+e^{-Mx_{0}}}\le\frac{1}{\delta x_{0}^{k}+\frac{k\delta}{M}x_{0}^{k-1}}.\label{eq:2.8}
\end{equation}

On the other hand, one easily has by using the simple inequality $e^{Mx_{0}}\ge Mx_{0}$
that
\[
\frac{M}{k\delta}\le x_{0}^{k-1}e^{Mx_{0}}\le\frac{1}{M^{k-1}}e^{\left(k-1\right)Mx_{0}}e^{Mx_{0}}\le\frac{1}{M^{k-1}}e^{kMx_{0}},
\]

which leads to
\begin{equation}
x_{0}\ge\frac{1}{kM}\ln\left(\frac{M^{k}}{k\delta}\right).\label{eq:2.9}
\end{equation}

Combining \eqref{eq:2.8} and \eqref{eq:2.9}, we obtain
\[
g\left(x\right)\le\frac{1}{\delta x_{0}^{k}}\le\frac{1}{\delta}\left(\frac{kM}{\ln\left(\frac{M^{k}}{k\delta}\right)}\right)^{k}.
\]
\end{proof}
\begin{lem}
\label{lem:2}Given $k\ge1$, let $\delta$ and $M$ be two positive
constants. Then for $0\le a\le M$, the following inequality holds:
\[
\frac{e^{-ax}}{\delta x^{k}+e^{-Mx}}\le\left(kM\right)^{k\left(1-\frac{a}{M}\right)}\delta^{\frac{a}{M}-1}\left(\ln\left(\frac{M^{k}}{k\delta}\right)\right)^{k\left(\frac{a}{M}-1\right)},\quad x>0.
\]
\end{lem}
\begin{proof}
Applying Lemma \ref{lem:1} and direct computations, we obtain
\begin{eqnarray*}
\frac{e^{-ax}}{\delta x^{k}+e^{-Mx}} & = & \frac{e^{-ax}}{\left(\delta x^{k}+e^{-Mx}\right)^{\frac{a}{M}}\left(\delta x^{k}+e^{-Mx}\right)^{1-\frac{a}{M}}}\\
 & \le & \frac{1}{\left(\delta x^{k}+e^{-Mx}\right)^{1-\frac{a}{M}}}\\
 & \le & \left[\left(kM\right)^{k}\delta^{-1}\left(\ln\left(\frac{M^{k}}{k\delta}\right)^{-k}\right)\right]^{1-\frac{a}{M}}\\
 & \le & \left(kM\right)^{k\left(1-\frac{a}{M}\right)}\delta^{\frac{a}{M}-1}\left(\ln\left(\frac{M^{k}}{k\delta}\right)\right)^{k\left(\frac{a}{M}-1\right)}.
\end{eqnarray*}
\end{proof}
\begin{lem}
\label{lem:3}For $\delta>0,k\ge1$, the filtering function $\Phi_{\bar{\mu},\lambda}^{\delta}$
given by \eqref{eq:6} can be bounded by
\[
\Phi_{\bar{\mu},\lambda}^{\delta}\left(s,t\right)\le\left(kTq\right)^{\frac{kp\left(s-t\right)}{qT}}\delta^{\frac{p\left(t-s\right)}{qT}}\left(\ln\left(\frac{\left(Tq\right)^{k}}{k\delta}\right)\right)^{\frac{-kp\left(s-t\right)}{qT}},\quad0\le t\le s\le T.
\]

Moreover, it leads to
\[
\Phi_{\bar{\mu},\lambda}^{\delta}\left(T,t\right)\le\left(kTq\right)^{\frac{kp\left(T-t\right)}{qT}}\delta^{\frac{p\left(t-T\right)}{qT}}\left(\ln\left(\frac{\left(Tq\right)^{k}}{k\delta}\right)\right)^{\frac{-kp\left(T-t\right)}{qT}},\quad0\le t\le T.
\]
\end{lem}
\begin{proof}
The proof is straightforward due to the assumption $\left(\mbox{A}_{2}\right)$
and replacing $x=\lambda,a=qT-p\left(s-t\right)$ and $M=Tq$ in Lemma
\ref{lem:2}.\end{proof}
\begin{rem}
We now turn to introduce the abstract Gevrey class of functions of
order $s>0$ and index $\sigma>0$, denoted by $\mathbb{G}_{\sigma}^{s}$,
(see, \emph{e.g.}, \cite{CRT99}) characterized as
\[
\mathbb{G}_{\sigma}^{s}:=\left\{ u\in\mathcal{H}:\int_{0}^{\infty}\lambda^{s}e^{2\sigma\lambda}d\left\Vert E_{\lambda}u\right\Vert ^{2}<\infty\right\} ,
\]

equipped with the norm
\[
\left\Vert u\right\Vert _{\mathbb{G}_{\sigma}^{s}}^{2}:=\int_{0}^{\infty}\lambda^{s}e^{2\sigma\lambda}d\left\Vert E_{\lambda}u\right\Vert ^{2}<\infty.
\]

Let us remind that the focusing point in this paper lies in considering
the Laplacian, i.e. $\mathcal{A}=-\Delta$. Here, the abstract Gevrey
class $\mathbb{G}_{\sigma}^{s}$ is defined as the space $\mathcal{D}\left(\left(-\Delta\right)^{s/2}e^{\sigma\left(-\Delta\right)^{1/2}}\right)$,
and thus meaningful with the inner product
\[
\left\langle u,v\right\rangle _{\mathcal{D}\left(\left(-\Delta\right)^{s/2}e^{\sigma\left(-\Delta\right)^{1/2}}\right)}:=\left\langle \left(-\Delta\right)^{s/2}e^{\sigma\left(-\Delta\right)^{1/2}}u,\left(-\Delta\right)^{s/2}e^{\sigma\left(-\Delta\right)^{1/2}}v\right\rangle ,
\]

for $u,v\in\mathcal{D}\left(\left(-\Delta\right)^{s/2}e^{\sigma\left(-\Delta\right)^{1/2}}\right)$.
Therefore, $\mathbb{G}_{\sigma}^{s}$ is the Hilbert space. 
\end{rem}

\section{Main results}

Observe that the representation \eqref{eq:7} of regularized solution
is naturally an integral equation. So what we have to consider in
this section are well-posedness of the equation and convergence results.
In the former work, it is easy to prove the existence and uniqueness
of such an integral equation. In fact, let us remark that Lemma \ref{lem:3}
says the boundedness of the filtering function $\Phi_{\bar{\mu},\lambda}^{\delta}$,
then with the aid of the fact that $f_{R^{\delta}}$ is globally Lipschitz
by \eqref{eq:2.5} we obtain the result by the well-known Banach fixed-point
theorem. Thereby, we shall claim the lemma below without proof.
\begin{lem}
\label{lem:exist}The integral equation $u^{\delta}\left(t\right)=\mathcal{G}\left(u^{\delta}\right)\left(t\right)$
where $\mathcal{G}$ maps from $C\left(\left[0,T\right];\mathcal{H}\right)$
into itself such that
\begin{equation}
\mathcal{G}\left(u^{\delta}\right)\left(t\right)=\int_{0}^{\infty}\Phi_{\bar{\mu},\lambda}^{\delta}\left(T,t\right)dE_{\lambda}u_{T}^{\delta}-\int_{t}^{T}\int_{0}^{\infty}\Phi_{\bar{\mu},\lambda}^{\delta}\left(s,t\right)dE_{\lambda}f_{R^{\delta}}\left(s,u^{\delta}\right)ds,\label{eq:26}
\end{equation}

has a unique solution $u^{\delta}\in C\left(\left[0,T\right];\mathcal{H}\right)$
for every $\delta>0$.\end{lem}
\begin{thm}
\label{thm:stable}If $u^{\delta}$ and $v^{\delta}$ are two solutions
of problem \eqref{eq:7} corresponding to the final observations $u_{T}^{\delta}$
and $v_{T}^{\delta}$ respectively, then for all $0\le t<T$,
\[
\left\Vert u^{\delta}\left(t\right)-v^{\delta}\left(t\right)\right\Vert \le\sqrt{2}\delta^{\frac{p}{q}\left(\frac{t}{T}-1\right)}\left(\ln\left(\frac{\left(Tq\right)^{k}}{k\delta}\right)\right)^{-\frac{kp}{q}\left(1-\frac{t}{T}\right)}\left(kTq\right)^{\frac{kp}{q}\left(1-\frac{t}{T}\right)}\mbox{exp}\left(C_{R^{\delta}}^{2}\left(T-t\right)^{2}\right)\left\Vert u_{T}^{\delta}-v_{T}^{\delta}\right\Vert .
\]
\end{thm}
\begin{proof}
From \eqref{eq:7} , one deduces that for all $0\le t<T$,
\begin{eqnarray*}
d^{\delta}\left(t\right) & = & u^{\delta}\left(t\right)-v^{\delta}\left(t\right)\\
 & = & \int_{0}^{\infty}\Phi_{\bar{\mu},\lambda}^{\delta}\left(T,t\right)dE_{\lambda}\left(u_{T}^{\delta}-v_{T}^{\delta}\right)-\int_{0}^{\infty}\int_{t}^{T}\Phi_{\bar{\mu},\lambda}^{\delta}\left(s,t\right)dE_{\lambda}\left(f_{R^{\delta}}\left(s,u^{\delta}\right)-f_{R^{\delta}}\left(s,v^{\delta}\right)\right)ds.
\end{eqnarray*}

Using Lemmas \ref{lem:3}, the elementary inequality $\left(a+b\right)^{2}\le2\left(a^{2}+b^{2}\right)$
and the global Lipschitz of $f_{R^{\delta}}$ by \eqref{eq:2.5},
we have
\[
\left\Vert d^{\delta}\left(t\right)\right\Vert ^{2}\le2\delta^{\frac{2p}{q}\left(\frac{t}{T}-1\right)}\left(\ln\left(\frac{\left(Tq\right)^{k}}{k\delta}\right)\right)^{-\frac{2kp}{q}\left(1-\frac{t}{T}\right)}\left(kTq\right)^{\frac{2kp}{q}\left(1-\frac{t}{T}\right)}\left\Vert u_{T}^{\delta}-v_{T}^{\delta}\right\Vert ^{2}+
\]
\begin{equation}
+2\delta^{\frac{2pt}{qT}}\left(\ln\left(\frac{\left(Tq\right)^{k}}{k\delta}\right)\right)^{\frac{2kpt}{qT}}\left(kTq\right)^{-\frac{2kpt}{qT}}C_{R^{\delta}}^{2}\left(T-t\right)\int_{t}^{T}\delta^{-\frac{2ps}{qT}}\left(\ln\left(\frac{\left(Tq\right)^{k}}{k\delta}\right)\right)^{-\frac{2pks}{qT}}\left(kTq\right)^{\frac{2kps}{qT}}\left\Vert d^{\delta}\left(s\right)\right\Vert ^{2}ds.\label{eq:10}
\end{equation}

Multiplying \eqref{eq:10} by $\delta^{-\frac{2pt}{qT}}\left(\ln\left(\frac{\left(Tq\right)^{k}}{k\delta}\right)\right)^{-\frac{2kpt}{qT}}\left(kTq\right)^{\frac{2kpt}{qT}}$
and putting
\[
w\left(t\right)=\delta^{-\frac{pt}{qT}}\left(\ln\left(\frac{\left(Tq\right)^{k}}{k\delta}\right)\right)^{-\frac{kpt}{qT}}\left(kTq\right)^{\frac{kpt}{qT}}\left\Vert d^{\delta}\left(t\right)\right\Vert ,
\]

we have
\[
w^{2}\left(t\right)\le2\delta^{-\frac{2p}{q}}\left(\ln\left(\frac{\left(Tq\right)^{k}}{k\delta}\right)\right)^{-\frac{2kp}{q}}\left(kTq\right)^{\frac{2kp}{q}}\left\Vert u_{T}^{\delta}-v_{T}^{\delta}\right\Vert ^{2}+2C_{R^{\delta}}^{2}\left(T-t\right)\int_{t}^{T}w^{2}\left(s\right)ds.
\]

Thanks to Gr\"onwall's inequality, we obtain
\[
w^{2}\left(t\right)\le2\delta^{-\frac{2p}{q}}\left(\ln\left(\frac{\left(Tq\right)^{k}}{k\delta}\right)\right)^{-\frac{2kp}{q}}\left(kTq\right)^{\frac{2kp}{q}}\mbox{exp}\left(2C_{R^{\delta}}^{2}\left(T-t\right)^{2}\right)\left\Vert u_{T}^{\delta}-v_{T}^{\delta}\right\Vert ^{2}.
\]

Hence, we conclude that
\[
\left\Vert u^{\delta}\left(t\right)-v^{\delta}\left(t\right)\right\Vert \le\sqrt{2}\delta^{\frac{p}{q}\left(\frac{t}{T}-1\right)}\left(\ln\left(\frac{\left(Tq\right)^{k}}{k\delta}\right)\right)^{-\frac{kp}{q}\left(1-\frac{t}{T}\right)}\left(kTq\right)^{\frac{kp}{q}\left(1-\frac{t}{T}\right)}\mbox{exp}\left(C_{R^{\delta}}^{2}\left(T-t\right)^{2}\right)\left\Vert u_{T}^{\delta}-v_{T}^{\delta}\right\Vert .
\]

This inequality follows the continuous dependence of the regularized
solution \eqref{eq:7} on $u_{T}^{\delta}$, and the theorem is proved.
\end{proof}
We have completed the stability result, and we now can say that our
constructed solution $u^{\delta}$ does uniquely exists by Lemma \ref{lem:exist}
and stable by Theorem \ref{thm:stable}. It remains to prove the convergence
coupled with the rate. The effective strategy is employing the approximate
solution corresponding to the exact data where we denote by $U^{\delta}$.
Generally speaking, the solution $U^{\delta}$ enables us, by the
triangle inequality, to deduce the main part of finding error estimate
to the exact solution $u$. Furthermore, as we know well, the current
tackled problem is widely general while our functional space in hands
is strictly limited, says $C\left(\left[0,T\right];\mathcal{H}\right)$
at most. Therefore, the next challenge lies in not only we have to
find an abstract space in such a way that it is capable of covering
our analysis, but also requires many efforts in computation. As a
consequence, the Gevrey regularity, as has been fully introduced,
will take place in our main results. The appearance of this regularity
is remarkably a concrete evidence for the difficulty of nonlinear
problems backward in time.
\begin{thm}
\label{thm:conv1}Suppose that the solution $u$ given by \eqref{eq:4}
of the problem \eqref{eq:1} belongs to $C\left(\left[0,T\right];\mathbb{G}_{\sigma}^{\gamma}\right)$
for $\sigma\ge Tq$ and $\gamma=2k$. Then for $\delta\in\left(0,1\right)$,
the error estimate over the space $\mathcal{H}$ between $u$ and
the functional $U^{\delta}$ defined by
\begin{equation}
U^{\delta}\left(t\right)=\int_{0}^{\infty}\Phi_{\bar{\mu},\lambda}^{\delta}\left(T,t\right)dE_{\lambda}u_{T}-\int_{t}^{T}\int_{0}^{\infty}\Phi_{\bar{\mu},\lambda}^{\delta}\left(s,t\right)dE_{\lambda}f_{R^{\delta}}\left(s,U^{\delta}\right)ds,\label{eq:U}
\end{equation}

is uniformly given by
\[
\left\Vert u\left(t\right)-U^{\delta}\left(t\right)\right\Vert \le\sqrt{2}\delta^{\frac{p\left(t-T\right)}{qT}+1}\left(\ln\left(\frac{\left(Tq\right)^{k}}{k\delta}\right)\right)^{-\frac{kp\left(T-t\right)}{qT}}\left(kTq\right)^{\frac{kp\left(T-t\right)}{qT}}\left\Vert u\right\Vert _{C\left(\left[0,T\right];\mathbb{G}_{\sigma}^{\gamma}\right)}\mbox{exp}\left(C_{R^{\delta}}^{2}\left(T-t\right)^{2}\right).
\]
\end{thm}
\begin{proof}
Observe \eqref{eq:4} and \eqref{eq:U}, they can be rewritten as
\begin{equation}
u\left(t\right)=\int_{0}^{\infty}e^{\bar{\mu}\left(0,T\right)\lambda}\left(e^{-\bar{\mu}\left(0,t\right)\lambda}dE_{\lambda}u_{T}-\int_{t}^{T}e^{\left(\bar{\mu}\left(t,s\right)-\bar{\mu}\left(0,T\right)\right)\lambda}dE_{\lambda}f\left(s,u\right)ds\right),\label{eq:3.4}
\end{equation}
\begin{equation}
U^{\delta}\left(t\right)=\int_{0}^{\infty}\frac{1}{\delta\lambda^{k}+e^{-\bar{\mu}\left(0,T\right)\lambda}}\left(e^{-\bar{\mu}\left(0,t\right)\lambda}dE_{\lambda}u_{T}-\int_{t}^{T}e^{\left(\bar{\mu}\left(t,s\right)-\bar{\mu}\left(0,T\right)\right)\lambda}dE_{\lambda}f_{R^{\delta}}\left(s,U^{\delta}\right)ds\right).\label{eq:3.5}
\end{equation}

Therefore, putting $d^{\delta}\left(t\right)=u\left(t\right)-U^{\delta}\left(t\right)$
one deduces from \eqref{eq:3.4} and \eqref{eq:3.5} that
\begin{eqnarray*}
d^{\delta}\left(t\right) & = & \int_{0}^{\infty}\left(e^{\bar{\mu}\left(0,T\right)\lambda}-\frac{1}{\delta\lambda^{k}+e^{-\bar{\mu}\left(0,T\right)\lambda}}\right)\left(e^{-\bar{\mu}\left(0,t\right)\lambda}dE_{\lambda}u_{T}-\int_{t}^{T}e^{\left(\bar{\mu}\left(t,s\right)-\bar{\mu}\left(0,T\right)\right)\lambda}dE_{\lambda}f\left(s,u\right)ds\right)\\
 &  & +\int_{t}^{T}\int_{0}^{\infty}\Phi_{\bar{\mu},\lambda}^{\delta}\left(s,t\right)dE_{\lambda}\left(f_{R^{\delta}}\left(s,U^{\delta}\right)-f\left(s,u\right)\right)ds\\
 & = & \int_{0}^{\infty}\delta\lambda^{k}\Phi_{\bar{\mu},\lambda}^{\delta}\left(T,t\right)\left(e^{\bar{\mu}\left(0,T\right)\lambda}dE_{\lambda}u_{T}-\int_{t}^{T}\int_{0}^{\infty}e^{\bar{\mu}\left(0,s\right)\lambda}dE_{\lambda}f\left(s,u\right)\right)\\
 &  & +\int_{t}^{T}\int_{0}^{\infty}\Phi_{\bar{\mu},\lambda}^{\delta}\left(s,t\right)dE_{\lambda}\left(f_{R^{\delta}}\left(s,U^{\delta}\right)-f\left(s,u\right)\right)ds\\
 & = & \int_{0}^{\infty}\delta\lambda^{k}\Phi_{\bar{\mu},\lambda}^{\delta}\left(T,t\right)e^{\bar{\mu}\left(0,t\right)\lambda}dE_{\lambda}u\left(t\right)+\int_{t}^{T}\int_{0}^{\infty}\Phi_{\bar{\mu},\lambda}^{\delta}\left(s,t\right)dE_{\lambda}\left(f_{R^{\delta}}\left(s,U^{\delta}\right)-f\left(s,u\right)\right)ds.
\end{eqnarray*}

Thanks to Lemma \ref{lem:3}, we thus have
\begin{eqnarray*}
\left|d^{\delta}\left(t\right)\right| & \le & \left(kTq\right)^{\frac{kp\left(T-t\right)}{qT}}\delta^{\frac{p\left(t-T\right)}{qT}+1}\left(\ln\left(\frac{\left(Tq\right)^{k}}{k\delta}\right)\right)^{-\frac{kp\left(T-t\right)}{qT}}\int_{0}^{\infty}\lambda^{k}e^{\bar{\mu}\left(0,t\right)\lambda}d\left|E_{\lambda}u\left(t\right)\right|\\
 &  & +\left(kTq\right)^{\frac{kp\left(s-t\right)}{qT}}\int_{t}^{T}\int_{0}^{\infty}\delta^{\frac{p\left(t-s\right)}{qT}}\left(\ln\left(\frac{\left(Tq\right)^{k}}{k\delta}\right)\right)^{-\frac{kp\left(s-t\right)}{qT}}d\left|E_{\lambda}\left(f_{R^{\delta}}\left(s,U^{\delta}\right)-f\left(s,u\right)\right)\right|ds.
\end{eqnarray*}

Combining this with $\left(\mbox{A}_{2}\right)$ shows that
\[
\delta^{-\frac{2pt}{qT}}\left(\ln\left(\frac{\left(Tq\right)}{k\delta}\right)\right)^{-\frac{2kpt}{qT}}\left(kTq\right)^{\frac{2kpt}{qT}}\left\Vert d^{\delta}\left(t\right)\right\Vert ^{2}\le2\delta^{2\left(1-\frac{p}{q}\right)}\left(kTq\right)^{\frac{2kp}{q}}\left(\ln\left(\frac{\left(Tq\right)^{k}}{k\delta}\right)\right)^{-\frac{2kp}{q}}\left\Vert u\right\Vert _{C\left(\left[0,T\right];\mathbb{G}_{\sigma}^{\gamma}\right)}^{2}+
\]
\begin{equation}
+2\left(T-t\right)\int_{t}^{T}\delta^{-\frac{2ps}{qT}}\left(kTq\right)^{\frac{2kps}{qT}}\left(\ln\left(\frac{\left(Tq\right)}{k\delta}\right)\right)^{-\frac{2kps}{qT}}\left\Vert f_{R^{\delta}}\left(s,U^{\delta}\right)-f\left(s,u\right)\right\Vert ^{2}ds.\label{eq:3.6}
\end{equation}

where we have used the inequality $\left(a+b\right)^{2}\le2\left(a^{2}+b^{2}\right)$
and H\"older's inequality, and multiplied both sides of the inequality
by $\delta^{-\frac{2pt}{qT}}\left(\ln\left(\frac{\left(Tq\right)}{k\delta}\right)\right)^{-\frac{2kpt}{qT}}\left(kTq\right)^{\frac{2kpt}{qT}}$. 

Notice from the fact that $\lim_{\delta\to0^{+}}R^{\delta}=\infty$,
one has, by the idea of \cite{TDM14}, there exists $\delta_{0}>0$
such that $f\left(t,u\right)=f_{R^{\delta}}\left(t,u\right)$ for
all $\delta\in\left(0,\delta_{0}\right)$. Using this fact in combination
with the globally Lipschitz condition by \eqref{eq:2.5}, it thus
follows from \eqref{eq:3.6} that
\begin{equation}
w^{2}\left(t\right)\le2\delta^{2\left(1-\frac{p}{q}\right)}\left(kTq\right)^{\frac{2kp}{q}}\left(\ln\left(\frac{\left(Tq\right)^{k}}{k\delta}\right)\right)^{-\frac{2kp}{q}}\left\Vert u\right\Vert _{\mathbb{G}_{\sigma}^{\gamma}}^{2}+2\left(T-t\right)C_{R^{\delta}}^{2}\int_{t}^{T}w^{2}\left(s\right)ds,\label{eq:3.7}
\end{equation}

where we have putted
\[
w\left(t\right)=\delta^{-\frac{pt}{qT}}\left(\ln\left(\frac{\left(Tq\right)}{k\delta}\right)\right)^{-\frac{kpt}{qT}}\left(kTq\right)^{\frac{2kpt}{qT}}\left\Vert d^{\delta}\left(t\right)\right\Vert .
\]

So now, it is clear that applying Gr\"onwall's inequality to \eqref{eq:3.7}
gives us
\[
w^{2}\left(t\right)\le2\delta^{2\left(1-\frac{p}{q}\right)}\left(kTq\right)^{\frac{2kp}{q}}\left(\ln\left(\frac{\left(Tq\right)^{k}}{k\delta}\right)\right)^{-\frac{2kp}{q}}\left\Vert u\right\Vert _{C\left(\left[0,T\right];\mathbb{G}_{\sigma}^{\gamma}\right)}^{2}\mbox{exp}\left(2C_{R^{\delta}}^{2}\left(T-t\right)^{2}\right).
\]

Hence, we obtain
\[
\left\Vert d^{\delta}\left(t\right)\right\Vert \le\sqrt{2}\delta^{\frac{p\left(t-T\right)}{qT}+1}\left(\ln\left(\frac{\left(Tq\right)^{k}}{k\delta}\right)\right)^{-\frac{kp\left(T-t\right)}{qT}}\left(kTq\right)^{\frac{kp\left(T-t\right)}{qT}}\left\Vert u\right\Vert _{C\left(\left[0,T\right];\mathbb{G}_{\sigma}^{\gamma}\right)}\mbox{exp}\left(C_{R^{\delta}}^{2}\left(T-t\right)^{2}\right),
\]

which leads to the result of the theorem.\end{proof}
\begin{thm}
\label{thm:conv2}Assume that the solution $u$ given by \eqref{eq:4}
of the problem \eqref{eq:1} has the same regularity in Theorem \ref{thm:conv1}.
Then for $\delta\in\left(0,1\right)$, the error estimate over the
space $\mathcal{H}$ between the functional $u^{\delta}$ in \eqref{eq:7},
and the exact solution $u$ is uniformly given by
\[
\left\Vert u\left(t\right)-u^{\delta}\left(t\right)\right\Vert \le C_{T}\delta^{\frac{p}{q}\left(\frac{t}{T}-1\right)+1}\left(\ln\left(\frac{\left(Tq\right)^{k}}{k\delta}\right)\right)^{-\frac{kp}{q}\left(1-\frac{t}{T}\right)}\left(kTq\right)^{\frac{kp}{q}\left(1-\frac{t}{T}\right)}\mbox{exp}\left(C_{R^{\delta}}^{2}\left(T-t\right)^{2}\right),
\]

where $C_{T}$ is a generic positive constant depending only on $T$.\end{thm}
\begin{proof}
Due to Theorem \ref{thm:stable}-\ref{thm:conv1} and the assumption
$\left(\mbox{A}_{3}\right)$, the proof is straightforward. Indeed,
using the triangle inequality we have the following estimate
\begin{eqnarray*}
\left\Vert u\left(t\right)-u^{\delta}\left(t\right)\right\Vert  & \le & \left\Vert u\left(t\right)-U^{\delta}\left(t\right)\right\Vert +\left\Vert U^{\delta}\left(t\right)-u^{\delta}\left(t\right)\right\Vert \\
 & \le & \sqrt{2}\delta^{\frac{p\left(t-T\right)}{qT}+1}\left(\ln\left(\frac{\left(Tq\right)^{k}}{k\delta}\right)\right)^{-\frac{kp\left(T-t\right)}{qT}}\left(kTq\right)^{\frac{kp\left(T-t\right)}{qT}}\left\Vert u\right\Vert _{C\left(\left[0,T\right];\mathbb{G}_{\sigma}^{\gamma}\right)}\mbox{exp}\left(C_{R^{\delta}}^{2}\left(T-t\right)^{2}\right)\\
 &  & +\sqrt{2}\delta^{\frac{p}{q}\left(\frac{t}{T}-1\right)}\left(\ln\left(\frac{\left(Tq\right)^{k}}{k\delta}\right)\right)^{-\frac{kp}{q}\left(1-\frac{t}{T}\right)}\left(kTq\right)^{\frac{kp}{q}\left(1-\frac{t}{T}\right)}\mbox{exp}\left(C_{R^{\delta}}^{2}\left(T-t\right)^{2}\right)\left\Vert u_{T}^{\delta}-u_{T}\right\Vert \\
 & \le & \sqrt{2}\delta^{\frac{p}{q}\left(\frac{t}{T}-1\right)+1}\left(\ln\left(\frac{\left(Tq\right)^{k}}{k\delta}\right)\right)^{-\frac{kp}{q}\left(1-\frac{t}{T}\right)}\left(kTq\right)^{\frac{kp}{q}\left(1-\frac{t}{T}\right)}\mbox{exp}\left(C_{R^{\delta}}^{2}\left(T-t\right)^{2}\right)\left(1+\left\Vert u\right\Vert _{C\left(\left[0,T\right];\mathbb{G}_{\sigma}^{\gamma}\right)}\right).
\end{eqnarray*}

So our objective is well-accomplished. This completes the proof of
the theorem.\end{proof}
\begin{rem}
Let us now sum up what we have obtained in Lemma \ref{lem:exist}
and Theorem \ref{thm:stable}-\ref{thm:conv2}. In principle, comparing
our stability estimate with previous works, it is better or rather
general than the previous orders, \emph{e.g.} in \cite{CO94,DB05,LL67,TQTT13-1,TBMN15,TT10}.
As might be expected, the order of the stability here is $\delta^{\frac{p}{q}\left(\frac{t}{T}-1\right)}\left(\ln\left(\frac{\left(Tq\right)^{k}}{k\delta}\right)\right)^{-\frac{kp}{q}\left(1-\frac{t}{T}\right)}$
containing the result in \cite[Theorem 2]{DA12} where $p=q=1$ are
treated.

Furthermore, the modified quasi-boundary value method has also been
proved that it is beneficial if one considers the error estimate in
the original time $t=0$. More rigorously, the result in Theorem \ref{thm:conv2}
says that
\[
\left\Vert u\left(0\right)-u^{\delta}\left(0\right)\right\Vert \le C_{T}\delta^{1-\frac{p}{q}}\left(\ln\left(\frac{\left(Tq\right)^{k}}{k\delta}\right)\right)^{-\frac{kp}{q}}\left(kTq\right)^{\frac{kp}{q}}\mbox{exp}\left(C_{R^{\delta}}^{2}T^{2}\right).
\]

It is thus evident to claim that the modified quasi-boundary value
proposed here, in general, not only gives a sharp and better approximation
than many well-known methods such as quasi-reversibility studied in
\cite{LL67}, stabilized quasi-reversibility in \cite{TDM14} and
quasi-boundary value investigated in \cite{CO94,DB05,Show83}, but
also proves the applicability of the method to a wide range of nonlinear
parabolic equations.
\end{rem}

\begin{rem}
It is worth noting that if the nonlinear source term $f\left(\cdot,u\right)$
and initial data $u\left(0\right)$ belong to the space $L^{2}\left(0,T;\mathbb{G}_{\sigma}^{\gamma}\right)$
and $\mathcal{D}\left(\mathcal{A}^{\gamma}\right)$ for $\sigma\ge Tq$
and $\gamma=2k$, respectively, we then obtain the same error estimate
in Theorem \ref{thm:conv2}. Additionally, it is comprehensive that
we may assume $\gamma\ge2k$ if $\lambda\ge1$. Moreover, it can be
proved in Theorem \ref{thm:conv2} that if further assume that $u\in C\left(\left[0,T\right];\mathbb{G}_{\sigma'}^{\gamma}\cap\mathbb{G}_{\sigma}^{\gamma}\right)$
for $\sigma'=Tp$ and $p<q$, the error can be extremely rigorously
estimated in the sense that
\[
C_{1}\delta^{\frac{p}{q}\left(\frac{t}{T}-1\right)+1}\left(\ln\left(\frac{\left(Tq\right)^{k}}{k\delta}\right)\right)^{-\frac{kp}{q}\left(1-\frac{t}{T}\right)}\le\left\Vert u\left(t\right)-u^{\delta}\left(t\right)\right\Vert \le C_{2}\delta^{\frac{p}{q}\left(\frac{t}{T}-1\right)+1}\left(\ln\left(\frac{\left(Tq\right)^{k}}{k\delta}\right)\right)^{-\frac{kp}{q}\left(1-\frac{t}{T}\right)},
\]

in which $C_{1}=C\left(k,p,q,T,\left\Vert u\right\Vert _{C\left(\left[0,T\right];\mathbb{G}_{\sigma'}^{\gamma}\right)}\right)$
and $C_{2}=C\left(k,p,q,T,R^{\delta},\left\Vert u\right\Vert _{C\left(\left[0,T\right];\mathbb{G}_{\sigma}^{\gamma}\right)}\right)$.
\end{rem}

\section{Conclusions and future challenges}

This is the moment to briefly review what we have done in this paper
and what remains to be done.

We have generalized the framework to the certain modified quasi-boundary
value method, and applied it to approximate the abstract backward
parabolic problem with time-dependent coefficients and local Lipschitzian
source. Under Gevrey regularity on the true solution, theoretical
considerations claim that our method is stable in $\mathcal{H}$ for
every noise level and uniformly convergent in $\mathbb{G}_{\sigma}^{\gamma}$.
The error estimate exhibits a faster (or rather general) rate than
the quasi-reversibility method and the quasi-boundary value method
with their modified versions. Furthermore, the error at the original
point is highly valuable. This work is an improvement for the results
obtained in, for example, \cite{DA12,DB05,TDMK15,TT10}.

Bearing in mind the monumental intellectual effort that went into
the last two decades of developing regularization methods, it would
have been surprising that we have been able to analyze in a considerably
more complicated framework. Nevertheless, much further investigation,
as we now indicate, is required. We thus wish to single out other
problems and challenges for future work. These include:
\begin{enumerate}
\item The local Lipschitzian source is obviously general, but in many biological
situations, nonlocal effects (\emph{e.g.} local movement of species)
are coupled with long-range influences, i.e. nonlinear partial integro-differential
equations must be considered. In physics, the well-known viscous Burgers'
equation had a very long history of establishment and development
in the studies of fluid mechanics and other related ramifications.
Even though this equation can be converted to the linear heat equation
by the Hopf-Cole transformation, it is more interesting to attack
on the original equation. In principle, it needs to take into consideration
the more realistic equation:
\[
\partial_{t}u-\mu\left(t,x,u,\int_{\Omega}udx\right)\Delta u=f\left(t,x,u,\nabla u,\int_{\Omega}udx\right).
\]

\item It seems that we may find a filtering function to obtain the regularized
problem. However, we have yet to see how well it works on a specific
model such as reaction-diffusion systems for two (or larger) species.
Besides, constructing a new truncated function in this case is a huge
challenge.
\item The initial-boundary value problem of a semi-linear parabolic equation
is somehow blow-up in finite time. It is either the intention or the
message to challenge ourselves this question for the backward problem.
All we need here is precisely providing the relevant theoretical details
and asymptotic behavior.
\item Another theoretical issue must be properly addressed stems from facing
a model with a periodically perforated domain which is getting more
attention in homogenization field in recent years. It is well-known
that there exists a family of eigenvectors and a corresponding family
of eigenvalues which have almost the same fundamental properties with
the ones in the common eigenvalue problem. In particular, those eigenvectors
are an orthonormal basis in $L^{2}$ over the considered perforated
domain, and their eigenvectors are increasing and tend to infinity
in a weak sense. For more details, the reader can be referred to \cite[Section 3.3]{CP99}.
As a consequence, a better understanding of how to treat this problem
is needed.
\item Concerning numerical issues, currently there is no rigorously numerical
analysis and treatment for computing the mild solution (attached by
the truncated function) even if it is typically linear. In the simple
case where eigenelements are explicitly known, one may try to use
the Filon method or many improvements/approaches (exotic quadrature,
Birkhoff expansions,...) to treat the highly oscillatory integrals
in the Fourier coefficients. However, among dozens of methods, we
need a great analyzing comparison. Since the backward problem is extremely
sensitive, we are not sure whether the ill-conditioned numerical problem
happens or not.
\item Computing the problem is, in general, expensive, and the cost is bigger
and bigger while solving with multiple dimensions. Consequently, efficient
algorithms in combination with parallel computing need to be developed.
On the other hand, it is not clear how errors associated with approximating
the mild solution would impact the stability of the regularized problem
of interest here. Recall that the filtering function has ``dangerous''
factors, ones increase so fast while the others reduces exponentially.
This may exacerbate such errors.\end{enumerate}
\begin{acknowledgement*}
The author would like to give many thanks to Dr. Nguyen Huy Tuan for
the whole-hearted guidance when the author was studying at Department
of Mathematics and Computer Science, Ho Chi Minh City University of
Science, Vietnam. The paper is also observed as his deep gratitude
to Department of Mathematics and Computer Science, Eindhoven University
of Technology in The Netherlands. Their hospitality is gratefully
acknowledged. He wishes to thank the handling editor as well as anonymous
referees for their most helpful comments on this paper.
\end{acknowledgement*}
\bibliographystyle{plain}
\bibliography{mybib}

\end{document}